\documentclass[12pt, reqno]{amsart}
\usepackage[utf8]{inputenc}	% Para caracteres en español
\usepackage{amssymb,amscd}
\usepackage{multirow,booktabs}
\usepackage[table]{xcolor}
\usepackage{fullpage}
\usepackage{lastpage}
\usepackage{fancyhdr}
\usepackage{mathrsfs}
\usepackage{wrapfig}
\usepackage{graphicx}
\usepackage{setspace}
\usepackage{calc}
\usepackage{bm}
\usepackage{tikz}
\usepackage{cite}
\usepackage{enumerate}
\usepackage{subfigure}
\usepackage{multicol}
\usepackage[colorlinks,
            linkcolor=blue,
            citecolor=red,
            ]{hyperref}
\usepackage{cancel}
\usepackage[all]{xy}
\usepackage{geometry}
\usepackage{empheq}
\usepackage{framed}
\usepackage{dsfont}
\usepackage{xcolor}
\theoremstyle{plain}
\newtheorem{theorem}{Theorem}[section]
\newtheorem{lemma}[theorem]{Lemma}
\newtheorem{proposition}[theorem]{Proposition}

\theoremstyle{remark}

\usepackage{indentfirst}
\begin{document}

\newcommand{\QQ}{\mathbb{Q}}
\newcommand{\RR}{\mathbb{R}}
\newcommand{\ZZ}{\mathbb{Z}}
\newcommand{\NN}{\mathbb{N}}
\newcommand{\N}{\mathbb{N}}
\newcommand{\Nor}{\mathscr{N}}
\newcommand{\CC}{\mathbb{C}}
\newcommand{\HH}{\mathbb{H}}
\newcommand{\EE}{\mathbb{E}}
\newcommand{\Var}{\operatorname{Var}}
\newcommand{\PP}{\mathbb{P}}
\newcommand{\Rd}{\mathbb{R}^d}
\newcommand{\Rn}{\mathbb{R}^n}
\newcommand{\BHH}{\overline{\mathbb{H}}}
\newcommand{\system}{(\Omega,\mathcal{F},\mu,T)}
\newcommand{\FF}{\mathcal{F}}
\newcommand{\GG}{\mathcal{G}}
\newcommand{\F}{\mathscr{F}}
\newcommand{\E}[1]{{\mathbb E}\!\left[#1\right]}
\newcommand{\cov}[1]{{\mathbb Cov}\!\left[#1\right]}
\newcommand{\MBS}{(\Omega,\mathcal{F})}
\newcommand{\MS}{(\Omega,\mathcal{F},\mu)}
\newcommand{\PS}{(\Omega,\mathcal{F},\mathbb{P})}
\newcommand{\Def}{\overset{\text{def}}{=}}
\newcommand{\ED}{\overset{\text{d}}{=}}
\newcommand{\Ser}[2]{#1_1,#1_2,\dots,#1_#2}
\newcommand{\convdis}{\overset{D}{\underset{n \to +\infty}{\longrightarrow}}}
\newcommand{\independent}{\perp\mkern-9.5mu\perp}
\def\avint{\mathop{\,\rlap{-}\!\!\int\!\!\llap{-}}\nolimits}
\newcommand{\convps}{\overset{a.s}{\underset{n \to +\infty}{\longrightarrow}}}
\newcommand{\convp}{\overset{\P}{\underset{n \to +\infty}{\longrightarrow}}}
\renewcommand{\P}{\mathbb{P}}

\author[Rafik Aguech]{Rafik Aguech}
 \address[Rafik Aguech]{Department of Statistics and Operation research, College of science, King Saud University, Riyadh, Saudi Arabia.}
  \email{raguech@ksu.edu.sa}

 \author[Shuo Qin]{Shuo Qin}
 \address[Shuo Qin]{Beijing Institute of Mathematical Sciences and Applications, and Yau Mathematical Sciences Center, Tsinghua University}
  \email{qinshuo@bimsa.cn}

\title{How two elephants can learn from each other}
\date{}

  \begin{abstract}
We consider a two-elephant walking model in which the elephants interact dynamically. At each time step, each elephant determines its next move randomly based on its partner's past movements. We show that the asymptotic behavior of the elephants mainly depends on the sign and the absolute value of the product of their reinforcement parameters. In various regimes, we establish the law of large numbers and the central limit theorem. Our proofs are based on a connection to the random recursive trees and employ stochastic approximation techniques and martingale methods.
  \end{abstract}

 
 
\maketitle

\section{Introduction and results}

The elephant random walk (ERW) was introduced by Schütz and Trimper \cite{schutz2004elephants} to investigate the influence of long-term memory on random walk dynamics. Specifically, the ERW is a nearest-neighbor random walk on $\ZZ$ where at each time step, the elephant selects one of its past steps uniformly at random; it then repeats that step with probability $p\in [0,1]$ or takes the opposite step with probability $1-p$. The parameter $p$ is referred to as the memory parameter. 

The ERW has attracted considerable research interest in recent years. We refer the reader to \cite{laulin2022elephant,MR4908097,MR4926011} and the references therein. In particular, it is known that the ERW exhibits three distinct regimes: diffusive, critical, and superdiffusive, depending on whether $p<3/4$, $p=3/4$ or $p>3/4$. In the diffusive and critical regimes, the law of large numbers, the law of the iterated logarithm, a strong invariance principle, and the asymptotic normality  (under appropriate normalization) of the ERW have been established, see \cite{baur2016elephant, MR3741953, MR3748931, MR3652225}. In the superdiffusive regime, it has been shown \cite{MR3741953,MR3652225} that, when the ERW is normalized by $n^{2p-1}$, it converges almost surely to a non-degenerate random variable, which is not Gaussian; however, the fluctuation around this limit is Gaussian \cite{MR4034803}. This limiting random variable has been studied further in \cite{guerin2023fixed, MR4926011}. 

In many real-life scenarios, individuals tend to base their decisions on the choices made by others—a phenomenon known as social learning. For example, a person may be more likely to visit a restaurant if it has been frequented by friends. Elephants, being highly social animals, exhibit similar behavior; they can observe and learn from each other’s actions. Motivated by the definition of the ERW, we propose a two-elephant walking model in which, at each time step, each elephant randomly selects one of its partner's previous steps and repeats that step with a given probability, taking the opposite step otherwise.

More precisely, let $(\xi_n^{(1)})_{n\geq 2}$ (resp. $(\xi_n^{(2)})_{n\geq 2}$) be i.i.d. Bernoulli random variables with parameter $p_1\in [0,1]$ (resp. parameter $p_2\in [0,1]$). Let $(u_n^{(1)})_{n\geq 1}$ and $(u_n^{(2)})_{n\geq 1}$ be independent random variables where each $u_n^{(i)}$ is uniformly distributed on the set $\{1,2,\dots, n\}$ for $i=1,2$. The two-elephant walking model $(S_n^{(1)}, S_n^{(2)})_{n\geq 0}$ and corresponding step sequences $(X_n^{(1)}, X_n^{(2)})_{n\geq 1}$ are defined as follows:
\begin{enumerate}[(i)]
  \item Assume that $S_0^{(1)} = S_0^{(2)} = 0$, and $S_1^{(1)} = X_1^{(1)}\in \{-1,1\}$ and $S_1^{(2)} = X_1^{(2)} \in \{-1,1\}$.
  \item For $n\geq 1$, given $(X_k^{(1)})_{1\leq k \leq n}$ and $(X_k^{(2)})_{1\leq k \leq n}$, set
  $$X_{n+1}^{(1)} := \xi_{n+1}^{(1)} X_{u_n^{(2)}}^{(2)} - (1-\xi_{n+1}^{(1)}) X_{u_n^{(2)}}^{(2)}, \quad \text{and} \quad X_{n+1}^{(2)} := \xi_{n+1}^{(2)} X_{u_n^{(1)}}^{(1)} - (1-\xi_{n+1}^{(2)}) X_{u_n^{(1)}}^{(1)}.$$
\end{enumerate}
Here $S_n^{(1)}$ and $S_n^{(2)}$ represent the positions of the first and second elephants at time $n$, respectively. We shall be interested in the asymptotic behavior of $S_n^{(1)}$ and $S_n^{(2)}$ as $n\to\infty$.

The constants $p_1$ and $p_2$ above are referred to as the memory parameters for the first and second elephants, respectively. For simplicity, we introduce the reinforcement parameters
 $$\alpha_1 := 2p_1 - 1, \quad \alpha_2 := 2p_2 - 1.$$

 Notice that if either $\alpha_1$ or $\alpha_2$ equals zero, the corresponding elephant does not learn from its partner and performs a simple random walk. Without loss of generality, we assume that $\alpha_1 \neq 0$ and $\alpha_2 = 0$ in such cases. The following Theorem \ref{thmsrw} concerns the asymptotic behavior of $S_n^{(1)}$; here, $C(\alpha_1)$ denote a positive constant depending on $\alpha_1$ whose value may change from line to line.

\begin{theorem}
\label{thmsrw}
Let $(S_n^{(1)}, S_n^{(2)})_{n\geq 0}$ be a two-elephant walking model with memory parameters $p_1 \in [0,1]$ and $p_2=1/2$.  \\
(i) We have the following law of the iterated logarithm-type estimate:
\begin{equation}
  \label{LILalpha0}
C(\alpha_1) \leq  \limsup_{n\to\infty} \frac{S_n^{(1)}}{\sqrt{2n \log \log n}} \leq 1+\sqrt{2}|\alpha_1|, \quad \text{a.s..}
\end{equation}
 In particular, the walk $S^{(1)}$ is recurrent in the sense that every integer is visited by $S^{(1)}$ infinitely often almost surely.  \\
(ii) We have the following central limit theorem:
$$
\frac{S_n^{(1)}}{\sqrt{(1+2\alpha_1^2)n}}\convdis \mathcal{N}\left(0, 1\right).
$$
where  we used $\convdis$ for convergence in distribution and $\mathcal{N}(0,1)$ denotes the standard normal distribution. 
\end{theorem}

Now assume that both $\alpha_1$ and $\alpha_2$ are non-zero. For $\alpha=(\alpha_1,\alpha_2)$, let 
$$\lambda_{\alpha}:=\begin{cases}
\operatorname{sgn}(\alpha_2)\sqrt{\alpha_1 \alpha_2}, & \text{if } \alpha_1 \alpha_2>0,\\ 
\mathrm{i}\operatorname{sgn}(\alpha_2)\sqrt{-\alpha_1 \alpha_2}, & \text{if } \alpha_1 \alpha_2<0,
\end{cases}
$$
where $\mathrm{i}$ denotes the imaginary unit. When $\alpha_1$ and $\alpha_2$ have the same sign, the following Theorem \ref{ThmtwoERW} shows that, as in the case of a single elephant random walk, there are three regimes: diffusive, critical and superdiffusive, depending on whether $|\lambda_{\alpha}|<1/2$, $|\lambda_{\alpha}|=1/2$ or $|\lambda_{\alpha}|>1/2$. The results are stated for $S^{(1)}$ only, the analogous results for $S^{(2)}$ can be obtained by interchanging $\alpha_1$ and $\alpha_2$. 

We let $C(\alpha)$ denote a positive constant depending on $\alpha_1$ and $\alpha_2$ whose value may change from line to line. Notice that in Theorem \ref{ThmtwoERW}, if $\alpha_1=\alpha_2=1$, we shall assume that $X_1^{(1)}\neq X_1^{(2)}$ since otherwise the model would be trivial with $S_n^{(1)}=S_n^{(2)}=n X_1^{(1)}$ for all $n\geq 1$. Similarly, if $\alpha_1=\alpha_2=-1$, we shall assume that $X_1^{(1)}= X_1^{(2)}$.

\begin{theorem}
  \label{ThmtwoERW}
  Let $(S_n^{(1)}, S_n^{(2)})_{n\geq 0}$ be a two-elephant walking model with memory parameters $p_1,p_2 \in [0,1]$ such that $\alpha_1 \alpha_2>0$. \\
  (i) If $|\lambda_{\alpha}|<1/2$, then we have the following LIL-type estimate
\begin{equation}
\label{Sn1LIL}
\limsup_{n\to \infty} \frac{|S_n^{(1)}|}{\sqrt{ n \log \log n}} \leq C(\alpha),
\end{equation}
and the asymptotic normality:
\begin{equation}
  \label{2ERWCLTaless12}
\frac{S_n^{(1)}}{\sqrt{n}} \convdis \mathcal{N}\left(0,\frac{1+2\alpha_1^2-2\alpha_1\alpha_2}{1-4\alpha_1 \alpha_2}\right).
\end{equation}
(ii) If $|\lambda_{\alpha}|=1/2$, then we have the following LIL-type estimate
$$
\limsup_{n\to \infty} \frac{|S_n^{(1)}|}{\sqrt{n \log n\log \log \log n}} \leq C(\alpha),
$$
and the asymptotic normality:
\begin{equation}
  \label{2ERWCLT12}
  \frac{S_n^{(1)}}{\sqrt{n \log n}} \convdis \mathcal{N}\left(0, \frac{\alpha_1+\alpha_2}{4\alpha_2}\right).
\end{equation}
(iii) If $|\lambda_{\alpha}|>1/2$, then we have the following a.s.-convergence result:
$$
\lim_{n\to \infty} \frac{S^{(1)}_n}{n^{|\lambda_{\alpha}|}} = \operatorname{sgn}(\alpha_2) \sqrt{\frac{\alpha_1}{\alpha_2}}\lim_{n\to \infty} \frac{S^{(2)}_n}{n^{|\lambda_{\alpha}|}}=W,  \quad a.s.,
$$
where $W$ is a non-degenerate random variable.
\end{theorem}

When $\lambda_{\alpha}$ is real with $|\lambda_{\alpha}|>1/2$, the asymptotic behavior of the fluctuation is given by the following theorem. 
\begin{theorem}
  \label{flucthm}
  In the setting of Theorem \ref{ThmtwoERW}, if $|\lambda_{\alpha}|\in (1/2,1)$, then we have 
$$
n^{|\lambda_{\alpha}|-\frac{1}{2}}\left(\frac{S_n^{(1)}}{n^{|\lambda_{\alpha}|}}-W\right)\convdis \mathcal{N}\left(0, \left(1+\frac{\alpha_1}{\alpha_2}\right)\frac{\sqrt{\alpha_1\alpha_2}}{4\alpha_1\alpha_2-1} \right);
$$
and if $|\lambda_{\alpha}|=1$, then we have 
$$
\sqrt{n}\left(\frac{S_n^{(1)}}{n^{|\lambda_{\alpha}|}}-W\right)\convdis L,
$$
where the random variable $L$ has characteristic function 
$$\EE (e^{\mathrm{i} t L})=\EE \exp\left(-\frac{(1-W^2)t^2}{3} \right), \quad t\in \RR.
$$
  \end{theorem}

When $\alpha_1$ and $\alpha_2$ have opposite signs, the situation is quite different. The following Theorem \ref{Thmal2difsin2ndm} shows that both $S^{(1)}$ and $S^{(2)}$ are always diffusive.

\begin{theorem}
  \label{Thmal2difsin2ndm}
  Let $(S_n^{(1)}, S_n^{(2)})_{n\geq 0}$ be a two-elephant walking model with memory parameters $p_1,p_2 \in [0,1]$ such that $\alpha_1\alpha_2<0$. Then, we have the following LIL-type estimate:
  \begin{equation}
\label{Sn1LILnega}
\limsup_{n\to \infty} \frac{|S_n^{(1)}|}{\sqrt{ n \log \log n}} \leq C(\alpha).
\end{equation}
Moreover, as $n\to \infty$,
\begin{equation}
  \label{2ERWCLTa12diff}
\frac{S_n^{(1)}}{\sqrt{n}} \convdis \mathcal{N}\left(0, \frac{1+2\alpha_1^2-2\alpha_1\alpha_2}{1-4  \alpha_1\alpha_2}\right), \quad \frac{S_n^{(2)}}{\sqrt{n}} \convdis \mathcal{N}\left(0, \frac{1+2\alpha_2^2-2\alpha_1\alpha_2}{1-4  \alpha_1\alpha_2}\right).
\end{equation}
\end{theorem}

\section{Proof of Theorem \ref{thmsrw}}

Let $(S_n^{(1)}, S_n^{(2)})_{n\geq 0}$ be a two-elephant walking model with memory parameters $p_1,p_2 \in [0,1]$. For any integer $n\geq 1$, we let $\FF_n:=\sigma(X_j^{(i)}: 1\leq j\leq n,\, i=1,\, 2)$. By definition, we have
\begin{equation}
  \label{Xn1cond}
\EE (X_{n+1}^{(1)} \mid \FF_n)= \frac{\alpha_1 S_n^{(2)}}{n},\quad \EE (X_{n+1}^{(2)} \mid \FF_n)=\frac{\alpha_2 S_n^{(1)}}{n}, \quad n\geq 1.
\end{equation}

\begin{proof}[Proof of Theorem \ref{thmsrw} (i)]
Using (\ref{Xn1cond}), we see that the following process $(M_n)_{n\geq 1}$ is a martingale with uniformly bounded differences with respect to $(\FF_n)_{n\geq 1}$:
$$
 M_n := S_n^{(1)} - \sum_{i=1}^{n-1}\frac{\alpha_1 S_i^{(2)}}{i}, \quad n\geq 1,
$$
where $\alpha_1:=2p_1-1 \in [-1,1]$. Notice that $S_n^{(2)}/n \to 0$ a.s. by the law of large numbers. Using (\ref{Xn1cond}) again, we have
$$
\EE ((M_{n+1}-M_n)^2 \mid \FF_n)= \EE \left(\left(X_{n+1}^{(1)}-\frac{\alpha_1 S_n^{(2)}}{n}\right)^2 \mid \FF_n\right) =1-\frac{\alpha_1^2 (S_n^{(2)})^2}{n^2},
$$
which implies that $\langle M_n \rangle/ n$ converges to $1$ a.s. as $n\to \infty$. By \cite[Corollary 4.1, Theorem 4.8]{MR0624435}, almost surely,
\begin{equation}
  \label{MnLIL}
\liminf_{n\to \infty}\frac{ M_n }{\sqrt{2n \log \log n}} =-1,\quad \limsup_{n\to \infty}\frac{ M_n }{\sqrt{2n \log \log n}} =1.
\end{equation}

We write $H_0:=0$ and $H_n:=\sum_{i=1}^{n} 1/i$ for $n\geq 1$. Then Abel's summation by parts formula gives
$$
 \sum_{i=1}^{n-1}\frac{S_i^{(2)}}{i}= \sum_{i=1}^{n-1}(H_{n-1}-H_{i-1})X_i^{(2)}.
$$
One can easily check that 
$$
\sum_{i=1}^{n-1} (H_{n-1}-H_{i-1})^2 \sim 2n, \quad \text{as } n\to \infty.
$$
Using \cite[Theorem 1]{MR647507} (with $a_{n,i}:=H_{n-1}-H_{i-1}$ for $1\leq i \leq n-1$, and $n_k, I_k$ being as in Equation (1.20) in the notation there), we obtain that almost surely,
\begin{equation}
  \label{LILS2divn}
\liminf_{n\to \infty}\frac{1}{\sqrt{4n \log \log n}}  \sum_{i=1}^{n-1}\frac{S_i^{(2)}}{i}  = -1, \quad \limsup_{n\to \infty}\frac{1}{\sqrt{4n \log \log n}}  \sum_{i=1}^{n-1}\frac{S_i^{(2)}}{i}  = 1.
\end{equation}
Combining this with (\ref{MnLIL}), we obtain that almost surely,
$$
 |1-|\alpha_1|\sqrt{2} | \leq  \limsup_{n\to\infty} \frac{S_n^{(1)}}{\sqrt{2n \log \log n}} \leq 1+|\alpha_1|\sqrt{2},
$$
which proves the first inequality in (\ref{LILalpha0}) for $\alpha_1\neq \pm \sqrt{2}/2$ and the second inequality in (\ref{LILalpha0}) for all $\alpha_1$.

We let $\GG_0:=\sigma( (X_j^{(2)})_{j\geq 1})$ and $\GG_n:=\sigma((X_j^{(1)})_{1\leq j\leq n}, (X_j^{(2)})_{j\geq 1})$ for $n\geq 1$. Then the process $(M_n)_{n\in \NN}$ is also a martingale with respect to $(\GG_n)_{n\geq 1}$. We denote 
$$
E^{(2)}:= \left\{\lim_{n\to \infty} \frac{S_n^{(2)}}{n}=0, \quad \limsup_{n\to \infty}\frac{1}{\sqrt{4n \log \log n}}  \sum_{i=1}^{n-1}\frac{S_i^{(2)}}{i}  = 1 \right\}.
$$
Notice that $\PP(E^{(2)})=1$ in view of (\ref{LILS2divn}), and $\langle M_n \rangle/ n$ converges to $1$ a.s. on the event $E^{(2)}$ as $n\to \infty$. Moreover, using that $|X_{n+1}^{(1)}|=1$ and $|S_n^{(2)}|\leq n$, one has
$$
\begin{aligned}
&\quad \ \EE ((M_{n+1}-M_n)^4 \mid \GG_n) \\
&= 1 - \frac{4\alpha_1 X_{n+1}^{(1)} S_n^{(2)}}{n}+ 6\alpha_1^2  \left(\frac{ S_n^{(2)}}{n}\right)^2 - 4 \alpha_1^3 X_{n+1}^{(1)}  \left(\frac{  S_n^{(2)}}{n}\right)^3 + \alpha_1^4\left(\frac{  S_n^{(2)}}{n}\right)^4 \\
&\leq 1+\frac{15 |S_n^{(2)}|}{n}.
\end{aligned}
$$
Therefore, for any $\varepsilon>0$, by the law of large numbers, one has, a.s. on $E^{(2)}$,
$$
\frac{1}{n} \sum_{k=1}^n \EE \left( (M_{k+1}-M_{k})^2 \mathds{1}_{\{|M_{k+1}-M_{k}|\geq \varepsilon \sqrt{n}\}} \mid \GG_k\right) \leq \frac{1}{\varepsilon^2 n^2 }\sum_{k=1}^n \EE ((M_{k+1}-M_{k})^4 \mid \GG_k) \to 0.
$$
Then \cite[Corollary 2.1.10]{MR1485774} implies that conditionally on $\GG_0$,
\begin{equation}
  \label{martMnCLTGo}
 \frac{M_n}{\sqrt{n}} \convdis \mathcal{N}\left(0, 1\right).
\end{equation}
Now fix $\varepsilon \in (0,1/2)$. On $E^{(2)}$, we can recursively define an infinite sequence of random variables $(\tau_j)_{j\in \NN}$ by setting $\tau_0:=0$ and 
$$
\tau_{j+1}:=\inf\left\{n>\tau_j: \frac{1}{\sqrt{4n \log \log n}}  \sum_{i=1}^{n-1}\frac{S_i^{(2)}}{i}  > 1-\varepsilon \right\}, \quad j\in \NN.
$$
We note that $(\tau_j)_{j\in \NN}$ are measurable with respect to $\GG_0$. Thus, by (\ref{martMnCLTGo}), for any $k \geq 1$,  
$$
\begin{aligned}
  &\quad \ \PP\left( \bigcup_{n\geq k} \left\{\frac{M_{\tau_n}}{\sqrt{2\tau_n \log \log \tau_n}} \geq -\varepsilon \right\} \mid \GG_0  \right) \mathds{1}_{E^{(2)}}\\ &\geq \lim_{n\to \infty} \PP\left( \frac{M_{\tau_n}}{\sqrt{\tau_n }} \geq -\varepsilon \sqrt{2\log \log \tau_n}  \mid \GG_0  \right) \mathds{1}_{E^{(2)}} =\mathds{1}_{E^{(2)}},
\end{aligned}
$$
which implies that 
$$
\PP\left( \bigcap_{k=1}\bigcup_{n\geq k} \left\{\frac{M_{\tau_n}}{\sqrt{2\tau_n \log \log \tau_n}} \geq -\varepsilon \right\} \mid \GG_0  \right) \mathds{1}_{E^{(2)}}=\mathds{1}_{E^{(2)}}.
$$
This shows that a.s. on $E^{(2)}$, there exist an infinite subsequence $(\tau_{n_k})_{k\in \NN}$ of $(\tau_n)_{n\in \NN}$ such that $M_{\tau_{n_k}}\geq -\varepsilon \sqrt{2\tau_{n_k} \log \log \tau_{n_k}}$, and in particular, if we assume that $\alpha_1=\sqrt{2}/2$, then almost surely on $E^{(2)}$,
$$
S_{\tau_{n_k}}^{(1)}=M_{\tau_{n_k}}+\sum_{i=1}^{\tau_{n_k}-1}\frac{ \sqrt{2} S_i^{(2)}}{2 i} \geq  (1-2\varepsilon)\sqrt{2\tau_{n_k} \log \log \tau_{n_k}},
$$
which proves (\ref{LILalpha0}) for $\alpha_1=\sqrt{2}/2$. The case $\alpha_1=-\sqrt{2}/2$ is proved similarly. 

By symmetry, one can show that almost surely,
$$
 \liminf_{n\to\infty} \frac{S_n^{(1)}}{\sqrt{2n \log \log n}} \leq -  C(\alpha_1), 
$$
and thus, almost surely, one has $\liminf S_n^{(1)}=-\infty$ and $\limsup S_n^{(1)}=\infty$, 
which implies that $S^{(1)}$ is recurrent.
\end{proof}

It is known that the ERW, or more generally, the step-reinforced random walk, has a connection to random recursive trees, see \cite{MR4237267, MR3827299, qin2024recurrence}. Similarly, we can relate $S^{(1)}$ with a random recursive tree and write $S^{(1)}_n$ as a randomly weighted sum. More precisely, we can give a second formulation of $S^{(1)}$ as follows:
\begin{itemize}
    \item Assume that $X_1^{(1)},X_1 \in \{-1,1\}$, and let $(X_n)_{n\geq 2}$ be i.i.d. Rademacher random variables with parameter $1/2$, which can be interpreted as the step sequence of the simple random walk $(S^{(2)}_n)_{n\geq 1}$.
    \item For each $n\geq k\geq 1$, let 
    $$
d_k(n):=\# \{j \in (k, n]: u_j^{(2)}=k\}
    $$
    with the convention that $d_n(n)=0$, where $(u_j^{(2)})_{j\geq 1}$ are independent random variables and $u_j^{(2)}$ is uniformly distributed on $\{1,2,\dots, j\}$. The random variable $d_k(n)$ counts the total number of times $S^{(1)}$ chooses $X_k$ (either repeat it or move in its reversed direction) up to time $n$. One can also interpret $(d_k(n))_{n\geq k}$ as the out-degrees of the vertex with label $k$ in a (growing) random recursive tree.
    \item Let $(\tilde{S}^{(k)})_{k\geq 1}$ be a sequence of independent biased random walks where for each $k\geq 1$, the random variables $\tilde{S}^{(k)}_1$, $\tilde{S}^{(k)}_2-\tilde{S}^{(k)}_1$, $\tilde{S}^{(k)}_3-\tilde{S}^{(k)}_2$, $\ldots$ are i.i.d. Rademacher random variables with parameter $p_1$. We use $\tilde{S}^{(k)}$ to determine the directions of $X_k$ in the step sequence $(X_n^{(1)})_{n\geq 2}$. More precisely, for $n\geq 1$, by a slight abuse of notation, we let
 \begin{equation}
  \label{defSn1secondway}
S_n^{(1)}:=X_1^{(1)}+\sum_{k=1}^n \tilde{S}^{(k)}_{d_n(k)} X_k.
 \end{equation}
 One can easily check by induction that $S^{(1)}$ defined by (\ref{defSn1secondway}) has the same distribution as the walk $S^{(1)}$ in Theorem \ref{thmsrw}. We shall not differentiate these two constructions of $S^{(1)}$ in the rest of the paper. We note that $(X_k)_{k\geq 1}$, $(d_n(k))_{n\geq k\geq 1}$ and $(\tilde{S}^{(k)})_{k\geq 1}$ are mutually independent.
\end{itemize}

We recall the following central limit theorem for martingale triangular arrays, see e.g. \cite[Corollary 3.1]{MR0624435}.

\begin{proposition}
\label{cltmartarray}
Suppose that for each $n\geq 1$, we have random variables $X_{n ,1}, \ldots, X_{n, n}$ on a probability space  $(\Omega, \mathcal{F}, \PP)$, with sub $\sigma$-fields $\mathcal{F}_{n,0} \subset \mathcal{F}_{n,1} \subset \cdots \subset \mathcal{F}_{n, n}$ of $\mathcal{F}$ such that $X_{n, k}$ is $\mathcal{F}_{n, k}$-measurable and $\EE\left(X_{n, k} \mid \mathcal{F}_{n, k-1}\right)=0$ a.s. for $k=1,2, \ldots, n$. Such an array is called a martingale triangular array. Let $S_n=X_{n, 1}+\cdots+X_{n, n}$ for $n\geq 1$. If 
\begin{equation}
\label{covinprob}
\sum_{k=1}^n \EE\left(X_{n, k}^2 \mid \mathcal{F}_{n, k-1}\right) \longrightarrow  1\quad \text{in probability as } n\to \infty,
\end{equation}
and for any $\varepsilon >0$,
\begin{equation}
\label{negcond}
\sum_{k=1}^n \EE\left(X_{n, k}^2 \mathds{1}_{\{|X_{n, k}| \geq \varepsilon\}} \mid \mathcal{F}_{n, k-1}\right) \longrightarrow 0\quad \text{in probability as } n\to \infty.
\end{equation}
Then $S_n$ converges in law as $n\to \infty$ to $\mathcal{N}(0,1)$. 
\end{proposition}

\begin{proposition}
  \label{semomsumconv}
As $n\to \infty$,
$$\lim_{n\to \infty}\frac{\sum_{k=1}^n (\tilde{S}^{(k)}_{d_n(k)})^2}{n(1+2\alpha_1^2)} =1 \quad \text{ in probability}.$$
\end{proposition}
The proof of Proposition \ref{semomsumconv} will be given later. Assuming Proposition \ref{semomsumconv}, we can now prove Theorem \ref{thmsrw} (ii).

\begin{proof}[Proof of Theorem \ref{thmsrw} (ii)]
We let $$\mathcal{F}_{n,0}:=\sigma(d_n(1),d_n(2),\dots, d_n(n), \tilde{S}^{(1)},\tilde{S}^{(2)},\dots,\tilde{S}^{(n)}),$$ 
and for $k=1,2,\dots,n$, we define
$$\mathcal{F}_{n,k}:=\sigma(d_n(1),d_n(2),\dots, d_n(n), \tilde{S}^{(1)},\tilde{S}^{(2)},\dots,\tilde{S}^{(n)}, X_1,X_2,\dots, X_k).$$
Clearly $\mathcal{F}_{n, 0} \subset \mathcal{F}_{n, 1} \subset \cdots \subset \mathcal{F}_{n, n}$ and $\tilde{S}^{(k)}_{d_n(k)} X_k$ is $\mathcal{F}_{n, k}$-measurable. Moreover, because of the independence, 
$$
\EE\left(\tilde{S}^{(k)}_{d_n(k)} X_k \mid \mathcal{F}_{n, k-1}\right)=\tilde{S}^{(k)}_{d_n(k)}\EE\left( X_k \mid \mathcal{F}_{n, k-1}\right)=0.
$$
Thus, for each $n$, $\left\{\tilde{S}^{(k)}_{d_n(k)}X_k, \mathcal{F}_{n k}, 1 \leqq k \leqq n\right\}$ is a sequence of martingale differences. We now check Conditions (\ref{covinprob}) and (\ref{negcond}) with 
$$X_{n,k}:=\frac{\tilde{S}^{(k)}_{d_n(k)}X_k}{\sqrt{n(1+2\alpha_1^2)}}, \quad k=1,2,\dots,n.$$ 
One can deduce (\ref{covinprob}) from Proposition \ref{semomsumconv}. Now let $d_n^{*}=\max\{d_n(1),d_n(2),\dots, d_n(n)\}$. Note that $|\tilde{S}^k_n|\leq n$ and $|X_k|=1$. One has
$$
\begin{aligned}
&\quad\ \sum_{k=1}^n \EE\left(\frac{(\tilde{S}^{(k)}_{d_n(k)}X_k)^2}{n(1+2\alpha_1^2)} \mathds{1}_{\left\{\left|\frac{\tilde{S}^{(k)}_{d_n(k)}X_k}{\sqrt{n(1+2\alpha_1^2)}}\right| \geq \varepsilon\right\}} \mid \mathcal{F}_{n, k-1}\right) \\
&\leq \sum_{k=1}^n \EE\left(\frac{(\tilde{S}^{(k)}_{d_n(k)})^2}{n(1+2\alpha_1^2)} \mathds{1}_{\left\{\frac{d_n^{*}}{\sqrt{n(1+2\alpha_1^2)}} \geq \varepsilon\right\}} \mid \mathcal{F}_{n, k-1}\right) \\
&= \mathds{1}_{\left\{\frac{d_n^{*}}{\sqrt{n(1+2\alpha_1^2)}} \geq \varepsilon\right\}} \sum_{k=1}^n \frac{(\tilde{S}^{(k)}_{d_n(k)})^2}{n(1+2\alpha_1^2)} \longrightarrow 0\quad \text{in probability as } n\to \infty,
\end{aligned}
$$
where in the last step we used that $d_n^{*}/ \log_2(n) \to 1$ in probability \cite{MR1346281}. This completes the proof of (\ref{negcond}). The desired result now follows from (\ref{defSn1secondway}) and Proposition \ref{cltmartarray}. 
\end{proof}

To prove Proposition \ref{semomsumconv}, we need the following two lemmas.

\begin{lemma}
  \label{secmosuS}
  $$\sum_{k=1}^{n}\EE (\tilde{S}^{(k)}_{d_n(k)})^2 = n(1+2\alpha_1^2)+ O(\log^2 n). $$
\end{lemma}

\begin{proof}
  For the biased random walk $\tilde{S}^{(1)}$, one has 
  $$
 \EE (\tilde{S}^{(1)}_n)^2= 4np_1(1-p_1) + n^2\alpha_1^2, \quad n\in \N.
  $$
 Thus, let $\GG_n:=\sigma(d_n(1),d_n(2),\dots,d_n(n))$. Using that $\sum_{k=1}^n d_n(k)=n-1$, one has 
 \begin{equation}
  \label{secmomsumS}
  \begin{aligned}
\sum_{k=1}^n\EE\left( (\tilde{S}^{(k)}_{d_n(k)})^2 \mid \GG_n\right) &= 4(n-1)p_1(1-p_1) + \alpha_1^2\sum_{k=1}^n d_n^2(k) \\
&=n-4p_1(1-p_1)-\alpha_1^2n + \alpha_1^2\sum_{k=1}^n d_n^2(k).  
\end{aligned}
 \end{equation}
 Recall that $d_k(n)$ is the out-degree of the vertex with label $k$ in a random recursive tree with $n$ vertices. It is known that, see e.g. \cite[Theorem 14.9]{MR3675279}, 
$$
\EE d_n(k)= H_{n-1}-H_{k-1}, \quad \operatorname{Var}(d_n(k))=H_{n-1}-H_{k-1}-\zeta_{n-1}(2)+\zeta_{k-1}(2),
$$
where $H_0:=0$ and $H_n:=\sum_{i=1}^{n} 1/i$ for $n\geq 1$, and 
$$
\zeta_0(2):=0,\quad \zeta_n(2)=\sum_{j=1}^n\frac{1}{j^2}, \quad n\geq 1.
$$
 Therefore, 
$$
\EE d^2_n(k)=(H_{n-1}-H_{k-1})^2+H_{n-1}-H_{k-1}-\zeta_{n-1}(2)+\zeta_{k-1}(2),\quad n\geq k\geq 1.
$$
Now observe that 
$$
 \log n-\log k= \int_k^n \frac{1}{x} dx \leq H_{n-1}-H_{k-1} \leq \int_{k}^{n-1} \frac{1}{x} dx +\frac{1}{k} =\log (n-1) -\log k +\frac{1}{k}.  
$$
and similarly.
$$
 \frac{1}{k}-\frac{1}{n} \leq  \zeta_{n-1}(2)-\zeta_{k-1}(2) \leq  \frac{1}{k}-\frac{1}{n-1}+ \frac{1}{k^2}.
$$
Using that
$$
0\leq \log (n-1) -\log k +\frac{1}{k}- (\log n-\log k) \leq \frac{1}{k},
$$
and 
$$
0\leq (\log (n-1) -\log k +\frac{1}{k})^2-(\log n-\log k)^2 \leq \frac{2(\log (n-1)-\log k)}{k} +\frac{1}{k^2},
$$
one can easily check that
$$
\sum_{k=1}^{n} \left(H_{n-1}-H_{k-1}-\log n+\log k\right) =O(\log n).
$$
and 
$$
\sum_{k=1}^{n} \left((H_{n-1}-H_{k-1})^2-(\log n-\log k)^2\right) =O(\log n).
$$
Consequently, 
$$
\sum_{k=1}^{n} \EE d^2_n(k)=  \sum_{k=1}^{n} (\log n-\log k)^2 + \sum_{k=1}^{n} (\log n-\log k) +O(\log n).
$$
One can then easily check that
\begin{equation}
\label{sumexped2}
\sum_{k=1}^{n} \EE d^2_n(k)=3n+O(\log^2 n).
\end{equation}
Combined with (\ref{secmomsumS}), this implies the desired result.
\end{proof}

\begin{lemma}
  \label{dnkinL1con}
  For $n\geq 1$, let 
$$T_n := \frac{1}{n} \sum_{k=1}^n d^2_n(k).$$
Then $(T_n)_{n\geq 1}$ converges to $3$ a.s. and in $L^1$.
\end{lemma}
\begin{proof}
  For $n> i \geq 0$, let 
$$Y_{n,i}:=\#\{d_n(k)=i: 1\leq k \leq n\}$$
be the number of vertices of out-degree $i$ in the growing random recursive tree at time $n$. Observe that
\[
T_n = \frac{1}{n} \sum_{i=1}^n i^2 Y_{n,i}, \quad n\geq 1.
\]
By \cite[Theorem 1.1]{MR2116576}, for any $i\geq 0$, $\lim_{n\to \infty}Y_{n,i}/n =2^{-i-1}$ a.s..  Note that $\sum_{i=1}^{\infty}i^22^{-i-1}=3$. For any $\varepsilon, \eta\in (0,1/2)$, we can choose a positive integer $K$ such that 
$$
\sum_{i=1}^{K}\frac{i^2}{2^{i+1}}\geq 3-\varepsilon \eta.
$$
Then, by the dominated convergence theorem, 
\begin{equation}
\label{domii2Yni}
\sum_{i=1}^K \lim_{n\to \infty} \frac{\EE i^2 Y_{n,i}}{n} =  \sum_{i=1}^{K}\frac{i^2}{2^{i+1}}\geq 3-\varepsilon\eta.
\end{equation}
 Then, 
$$
\begin{aligned}
\PP(|T_n-3|\geq \varepsilon) &\leq \PP(T_n\leq 3-\varepsilon)+\PP(T_n\geq 3+ \varepsilon) \\
&\leq \PP\left(\sum_{i=1}^{K}\frac{i^2Y_{n,i}}{n}\leq 3-\varepsilon\right)+ \PP\left(\sum_{i=1}^{K}\frac{i^2Y_{n,i}}{n}\geq 3\right)+\PP\left(\sum_{i=K+1}^{n}\frac{i^2Y_{n,i}}{n}\geq \varepsilon\right).
\end{aligned}
$$
In the second line, the first two terms converge to $0$ as $n\to \infty$ by the a.s.-convergence of $\sum_{i=1}^Ki^2Y_{n,i}/n \to \sum_{i=1}^Ki^22^{-i-1} \in [3-\varepsilon\eta, 3)$. For the last term, one can use (\ref{sumexped2}), (\ref{domii2Yni}) and Chebyshev's inequality to see that
$$
\lim_{n\to \infty}\PP\left(\sum_{i=K+1}^{n}\frac{i^2Y_{n,i}}{n}\geq \varepsilon\right) \leq  \lim_{n\to \infty}\frac{\EE T_n-\sum_{i=1}^K  \frac{\EE i^2 Y_{n,i}}{n}}{\varepsilon} \leq \eta.
$$
Since $\eta$ is arbitrary, this shows that $T_n \to 3$ in probability. Since $\EE T_n \to 3$, by Scheffé's lemma, see e.g. \cite[Theorem 5.12]{MR4226142}, one has $T_n$ converges to $3$ in $L^1$. 
\end{proof}

We are now ready to prove Proposition \ref{semomsumconv}.

\begin{proof}[Proof of Proposition \ref{semomsumconv}]
 For each biased random walk $\tilde{S}^{(k)}$, it is direct to check that 
\begin{equation}
\label{VartildSn}
\operatorname{Var}((\tilde{S}^{(k)}_n)^2)= 2 n(n-1)\left[1+2(n-2)\alpha_1^2-(2 n-3)\alpha_1^4\right].
\end{equation}
Therefore, conditionally on $\GG_n=\sigma((d_n(k))_{1\leq k \leq n})$, the variance of the sum $\sum_{k=1}^n (\tilde{S}^{(k)}_{d_n(k)})^2$ satisfies
\begin{equation}
  \label{VartildSncondbd}
\sum_{k=1}^n \operatorname{Var}((\tilde{S}^{(k)}_{d_n(k)})^2 \mid \GG_n) \leq C_1(\alpha_1)n+ C_2(\alpha_1) \sum_{k=1}^n d_n^2(k) + C_3(\alpha_1)\sum_{k=1}^n d_n^3(k),
\end{equation}
where $C_1(\alpha_1)$, $C_2(\alpha_1)$ and $C_3(\alpha_1)$ are positive constants depending on $\alpha_1$. 

Recall that $d_n^{*}=\max\{d_n(1),d_n(2),\dots, d_n(n)\}$ and $d_n^{*}/ \log_2(n) \to 1$ in probability. Fix $\varepsilon,\eta\in (0,1/3)$, and define the event
$$E_n(\varepsilon):=\left\{|\sum_{k=1}^n d_n^2(k) - 3n| \leq \varepsilon n \text{ and } d_n^{*} \leq (1+\varepsilon)\log_2 (n)\right\}.$$
 By Lemma \ref{dnkinL1con}, we have $\PP(E_n(\varepsilon))\geq 1-\eta$ for sufficiently large $n$. On the event $E_n(\varepsilon)$, by (\ref{secmomsumS}), one has
$$
\left|\sum_{k=1}^n\EE\left( (\tilde{S}^{(k)}_{d_n(k)})^2 \mid \GG_n\right) -n(1+2\alpha_1^2)\right| \leq  \alpha_1^2|\sum_{k=1}^n d_n^2(k)-3n| + 4p_1(1-p_1) \leq \varepsilon n +4p_1(1-p_1). 
$$
Therefore, for large $n$ such that $\varepsilon n \geq 4p_1(1-p_1)$, one has
$$
\begin{aligned}
&\quad\ \PP\left( \left|\sum_{k=1}^n (\tilde{S}^{(k)}_{d_n(k)})^2 - n(1+2\alpha_1^2)\right| \geq 3\varepsilon n \mid \GG_n\right)\mathds{1}_{E_n(\varepsilon)} \\
 &\leq \PP\left(\left|\sum_{k=1}^n (\tilde{S}^{(k)}_{d_n(k)})^2 - \sum_{k=1}^n\EE\left( (\tilde{S}^{(k)}_{d_n(k)})^2 \mid \GG_n\right)\right| \geq \varepsilon n\mid \GG_n \right)\mathds{1}_{E_n(\varepsilon)} \\
&\leq \frac{1}{\varepsilon^2 n^2}\sum_{k=1}^n \operatorname{Var}((\tilde{S}^{(k)}_{d_n(k)})^2 \mid \GG_n) \mathds{1}_{E_n(\varepsilon)} \leq \frac{C(\alpha_1,\varepsilon)\log n}{n}, 
\end{aligned}
$$
where $C(\alpha_1,\varepsilon)$ is a positive constant depending on $\alpha_1$ and $\varepsilon$, and we used (\ref{VartildSncondbd}) and that $d_n^3(k)\leq d_n^{*}d_n^2(k)$ in the third inequality. Taking the expectation over $\GG_n$, we obtain that 
$$
\begin{aligned}
  &\quad\ \PP\left( \left|\sum_{k=1}^n (\tilde{S}^{(k)}_{d_n(k)})^2 - n(1+2\alpha_1^2)\right| \geq 3\varepsilon n \right)  \\
&\leq \PP(E^c_n(\varepsilon)) + \PP\left( \left|\sum_{k=1}^n (\tilde{S}^{(k)}_{d_n(k)})^2 - n(1+2\alpha_1^2)\right| \geq 3\varepsilon n, E_n(\varepsilon) \right) \\
& \leq \eta+ \frac{C(\alpha_1,\varepsilon)\log n}{n}.
\end{aligned}
$$
By first letting $n\to \infty$ and then $\eta \to 0$, we obtain that $\sum_{k=1}^n (\tilde{S}^{(k)}_{d_n(k)})^2/n$ converges to $1+2\alpha_1^2$ in probability.
\end{proof}

\section{Proofs of Theorem \ref{ThmtwoERW} and Theorem \ref{Thmal2difsin2ndm}}  
\label{secthm2dERW}

Throughout this section, we let $(S_n^{(1)}, S_n^{(2)})_{n\geq 0}$ be a two-elephant walking model with memory parameters $p_1,p_2 \in [0,1]$. Recall that $\alpha_i=2p_i-1$ for $i=1,2$. We assume that the memory parameters are such that $\alpha_1\alpha_2 \neq 0$. 

Let us first introduce some preliminary notations. For $\alpha=(\alpha_1,\alpha_2)$, we define
$$
r_{\alpha}:=\begin{cases}
\sqrt{\frac{\alpha_1}{\alpha_2}}, & \text{if } \alpha_1\alpha_2>0, \\
\mathrm{i}\sqrt{\frac{-\alpha_1}{\alpha_2}}, & \text{if } \alpha_1\alpha_2<0.
\end{cases}
$$
We note that $\lambda_{\alpha}=r_{\alpha} \alpha_2$. When $\lambda_{\alpha} \in \{-1,1\}$ (which occurs if and only if $\alpha_1=\alpha_2=1$ or $\alpha_1=\alpha_2=-1$), we let
$$
\hat{\beta}_n=\frac{2}{n(n-1)}, \quad n\geq 2.
$$
When $\lambda_{\alpha} \notin \{-1,1\}$, for $n\geq 1$, we let 
\begin{equation}
  \label{defgammabeta}
\gamma_n(\pm \lambda_{\alpha}):=\frac{1\pm \lambda_{\alpha}}{n+1}, \quad \text{and}\quad \beta_n(\pm \lambda_{\alpha}):=\prod_{k=1}^{n-1}(1-\gamma_k(\pm \lambda_{\alpha}))=\frac{\Gamma(n\mp \lambda_{\alpha}) }{\Gamma(1\mp \lambda_{\alpha}) \Gamma(n+1)},
\end{equation}
with the convention that $\beta_1(\pm \lambda_{\alpha}):=1$. Notice that 
\begin{equation}
\label{betanasy}
\lim _{n \rightarrow \infty} \beta_n(\pm \lambda_{\alpha}) n^{1\pm \lambda_{\alpha}}=\frac{1}{\Gamma(1\mp \lambda_{\alpha})}.
\end{equation} 
For $n\geq 1$, we define 
\begin{equation}
  \label{defxyn}
x_n:=\frac{S_n^{(1)}-r_{\alpha} S_n^{(2)}}{n}, \quad y_n:=\frac{S_n^{(1)}+r_{\alpha} S_n^{(2)}}{n},
\end{equation}
and in particular, 
\begin{equation}
  \label{Sn1Sn2xny}
S_n^{(1)}=\frac{n}{2}(x_n+y_n), \quad S_n^{(2)}=\frac{n}{2r_{\alpha}}(y_n-x_n).
\end{equation}
We shall prove Theorem \ref{ThmtwoERW} and Theorem \ref{Thmal2difsin2ndm} by studying the asymptotic behaviors of $x_n$ and $y_n$. The following Lemmas \ref{lemmaxyn} and \ref{lemlambda1} show that $x_n$ and $y_n$ can be expressed as weighted sums of martingale differences. Recall that $(\FF_n)_{n\geq 0}$ is the natural filtration of $(S_n^{(1)}, S_n^{(2)})_{n\geq 0}$.

\begin{lemma}
\label{lemmaxyn}
If $\lambda_{\alpha} \neq \pm 1$, then for $n\geq 1$, the two random variables $x_n$ and $y_n$ defined in (\ref{defxyn}) can be expressed as
  \begin{equation}
    \label{xnynformula}
x_n=\beta_n(\lambda_{\alpha})\left(x_1+\sum_{j=1}^{n-1} \frac{\gamma_j(\lambda_{\alpha})}{\beta_{j+1}(\lambda_{\alpha})} \varepsilon_{j+1}^{(x)}\right), \quad y_n=\beta_n(-\lambda_{\alpha})\left(y_1+\sum_{j=1}^{n-1} \frac{\gamma_j(-\lambda_{\alpha})}{\beta_{j+1}(-\lambda_{\alpha})} \varepsilon_{j+1}^{(y)}\right),
  \end{equation}
where $(\varepsilon_{j+1}^{(x)})_{j\geq 1}$ and $(\varepsilon_{j+1}^{(y)})_{j\geq 1}$ are given by (\ref{defvarexy}) below, which are martingale difference sequences with respect to the filtration $(\FF_{j+1})_{j\geq 1}$.
\begin{equation}
    \label{defvarexy}
    \begin{aligned}
\varepsilon_{j+1}^{(x)}:&= \frac{1}{1+\lambda_{\alpha}}\left(X_{j+1}^{(1)}-r_{\alpha} X_{j+1}^{(2)}+\lambda_{\alpha} x_j\right), \quad j\geq 1, \\
    \varepsilon_{j+1}^{(y)}:&= \frac{1}{1-\lambda_{\alpha}}\left(X_{j+1}^{(1)}+r_{\alpha} X_{j+1}^{(2)}-\lambda_{\alpha} y_j\right), \quad j\geq 1.
    \end{aligned}
\end{equation}
\end{lemma}
\begin{proof}
  By definition, for any $n\geq 1$,
\begin{equation}
\label{recudeltxn}
    \begin{aligned}
    x_{n+1} - x_n&=\frac{ S_n^{(1)}-r_{\alpha} S_n^{(2)}+X_{n+1}^{(1)}-r_{\alpha} X_{n+1}^{(2)}}{n+1}-\frac{1}{n+1}\left(1+\frac{1}{n}\right)( S_n^{(1)}-r_{\alpha} S_n^{(2)})\\
    &=\frac{1}{n+1}\left(-x_n+ X_{n+1}^{(1)}-r_{\alpha} X_{n+1}^{(2)}\right) = \frac{1+\lambda_{\alpha}}{n+1}\left(-x_n+\varepsilon_{n+1}^{(x)}\right),
\end{aligned}
\end{equation}
where we used (\ref{Xn1cond}) in the third equality. One can then deduce the first equality in (\ref{xnynformula}) from (\ref{recudeltxn}) by induction. The expression for $y_n$ can be proved similarly.
\end{proof}

The introduction of the sequences $(x_n)_{n\geq 1}$ and $(y_n)_{n\geq 1}$ is motivated by the following observation: Define $z_n:=(S_n^{(1)}/n,S_n^{(2)}/n)$ for $n\geq 1$ (viewing these as column vectors), then by arguments analogous to those used in (\ref{recudeltxn}), we obtain
  \begin{equation}
    \label{stocapprozn}
z_{n+1}-z_n=\frac{1}{n+1}(Az_n + \varepsilon^{(z)}_{n+1}), \quad n\geq 1,
  \end{equation}
  where $(\varepsilon^{(z)}_{n+1})_{n\geq 1}$ is a martingale difference sequence and $A$ is the $2\times 2$ matrix given by
  $$
   A=\left(\begin{array}{ll}
         -1 & \alpha_2 \\
          \alpha_1 & -1
         \end{array}\right).
$$
Note that if $\alpha_1\alpha_2\neq 0$, then $A$ have two distinct eigenvalues $-1-\lambda_{\alpha}$ and  $-1+\lambda_{\alpha}$ with corresponding eigenvectors $(1,-r_{\alpha})$ and $(1,r_{\alpha})$, respectively. The iterative algorithm (\ref{stocapprozn}) is a Robbins–Monro algorithm, whose asymptotic behavior is closely related to that of its corresponding deterministic dynamical system (see e.g. \cite[Section 4]{MR1767993}). In our context, this deterministic system is given by
$$
\frac{d z(t)}{dt} =A z(t), \quad z(0)\in [0,1]^2. 
$$
Thus, it is not surprising that the two-elephant walking model exhibits distinct behaviors depending on whether $\lambda_{\alpha}$ is real or purely imaginary. By investigating the system along the eigenvector directions, that is, along the sequences $(x_n)_{n\geq 1}$ and $(y_n)_{n\geq 1}$, we are able to apply one-dimensional martingale techniques.

Lemma \ref{lemmaxyn} motivates us to study the following two martingales
\begin{equation}
  \label{defMxy}
M_n^{(x)}:=\sum_{j=1}^{n-1} \frac{\gamma_j(\lambda_{\alpha})}{\beta_{j+1}(\lambda_{\alpha})} \varepsilon_{j+1}^{(x)}, \quad M_n^{(y)}:=\sum_{j=1}^{n-1} \frac{\gamma_j(-\lambda_{\alpha})}{\beta_{j+1}(-\lambda_{\alpha})} \varepsilon_{j+1}^{(y)}, \quad n\geq 1,
\end{equation}
with the convention that $M_1^{(x)}=M_1^{(y)}=0$.

\begin{lemma}
  \label{asymvarxy}
  In the setting of Lemma \ref{lemmaxyn}, almost surely, $\lim_{n\to \infty} x_n=\lim_{n\to \infty} y_n=0$, and consequently, 
$$
\lim_{n\to \infty}\frac{S_n^{(1)}}{n}=\lim_{n\to \infty}\frac{S_n^{(2)}}{n}=0, \quad \text{almost surely}.
$$
Moreover, almost surely,
$$\lim_{j\to \infty} (1+\lambda_{\alpha})^2 \EE((\varepsilon_{j+1}^{(x)})^2 \mid \FF_j)= \lim_{j\to \infty} (1-\lambda_{\alpha})^2 \EE((\varepsilon_{j+1}^{(y)})^2 \mid \FF_j)= 1+\frac{\alpha_1}{\alpha_2},$$
and 
$$\lim_{j\to \infty} (1-\alpha_1\alpha_2)\EE(\varepsilon_{j+1}^{(x)}\varepsilon_{j+1}^{(y)} \mid \FF_j)=1-\frac{\alpha_1}{\alpha_2}.$$ 
\end{lemma}
\begin{proof}
Observe that $(\varepsilon_{j+1}^{(x)})_{j\geq 1}$ and $(\varepsilon_{j+1}^{(y)})_{j\geq 1}$ defined in (\ref{defvarexy}) are uniformly bounded. Thus, when $\lambda_{\alpha}$ is real (that is, $\alpha_1\alpha_2>0$), the quadratic variations of
 the martingales $M_n^{(x)}$ and $M_n^{(y)}$ satisfy 
\begin{equation}
  \label{quadvarMxyasym}
\begin{aligned}
  \langle M^{(x)} \rangle_n &= \sum_{j=1}^{n-1} \left(\frac{\gamma_j(\lambda_{\alpha})}{\beta_{j+1}(\lambda_{\alpha})}\right)^2 \EE((\varepsilon_{j+1}^{(x)})^2 \mid \FF_j) \leq C(\alpha) \sum_{j=1}^{n-1} j^{2\lambda_{\alpha}}, \\
 \langle M^{(y)} \rangle_n &= \sum_{j=1}^{n-1} \left(\frac{\gamma_j(-\lambda_{\alpha})}{\beta_{j+1}(-\lambda_{\alpha})}\right)^2 \EE((\varepsilon_{j+1}^{(y)})^2 \mid \FF_j) \leq C(\alpha) \sum_{j=1}^{n-1} j^{-2\lambda_{\alpha}},
\end{aligned}
\end{equation}
where $C(\alpha)$ is a positive constant, and we used (\ref{betanasy}) in the last step. By the law of large numbers for martingales, see e.g. \cite[Theorem 1.3.15]{MR1485774}, we have, almost surely,
$$
\lim_{n\to \infty}|x_n|=\lim_{n\to \infty}|\beta_n(\lambda_{\alpha})M^{(x)}_n|\leq \lim_{n\to \infty} n^{-1-\lambda_{\alpha}} \langle M^{(x)} \rangle_n^{2/3}=0,
$$
and similarly, $\lim_{n\to \infty}|y_n|=0$ almost surely. When $\lambda_{\alpha}$ is purely imaginary (that is, $\alpha_1\alpha_2<0$), one can consider the real and imaginary parts of the martingales $M_n^{(x)}$ and $M_n^{(y)}$ separately. One can then use the same argument to show that $\lim_{n\to \infty}|x_n|=\lim_{n\to \infty}|y_n|=0$ almost surely. This proves the first desired result.

By (\ref{defvarexy}), we have, for $j\geq 1$, 
\begin{equation}
\label{condvarxnyn}
\begin{aligned}
  \EE((\varepsilon_{j+1}^{(x)})^2 \mid \FF_j)&=\frac{1}{(1+\lambda_{\alpha})^2}\left(1+\frac{\alpha_1}{\alpha_2}-\alpha_1\alpha_2 x_j^2-\frac{2 r_{\alpha}\alpha_1\alpha_2 S_j^{(1)}S_j^{(2)} }{j^2}\right), \\
\EE((\varepsilon_{j+1}^{(y)})^2 \mid \FF_j)&=\frac{1}{(1-\lambda_{\alpha})^2}\left(1+\frac{\alpha_1}{\alpha_2}-\alpha_1\alpha_2 y_j^2+\frac{2 r_{\alpha}\alpha_1\alpha_2 S_j^{(1)}S_j^{(2)} }{j^2}\right), \\
\EE(\varepsilon_{j+1}^{(x)}\varepsilon_{j+1}^{(y)} \mid \FF_j)&=\frac{1}{1-\alpha_1\alpha_2}\left(1-\frac{\alpha_1}{\alpha_2}+\alpha_1\alpha_2 x_j y_j\right),
\end{aligned}
\end{equation}
where we used (\ref{Xn1cond}) and that $X_{j+1}^{(1)}$ and $X_{j+1}^{(2)}$ are independent conditionally on $\FF_j$. The desired limits then follow from the almost sure convergence of $x_j$ and $y_j$ to $0$ as $j\to \infty$.
\end{proof}

The proof of the following Lemma \ref{lemlambda1} is similar to that of Lemma \ref{lemmaxyn} and we omit it here. 
\begin{lemma}
  \label{lemlambda1}
  (i) If $\alpha_1=\alpha_2=-1$, then $(x_n)_{n\geq 1}$ is a bounded martingale and thus, converges a.s.. Moreover, for $n\geq 2$,
  $$
y_n=\hat{\beta}_n\left(y_2+\sum_{j=2}^{n-1} j \varepsilon_{j+1}^{(y)}\right), \quad n\geq 2.
$$
where $(\varepsilon_{j+1}^{(y)})_{j\geq 1}$ is defined in (\ref{defvarexy}). \\
(ii) If $\alpha_1=\alpha_2=1$, then $(y_n)_{n\geq 1}$ is a bounded martingale and thus, converges a.s.. Moreover, for $n\geq 2$,
$$
x_n=\hat{\beta}_n\left(x_2+\sum_{j=2}^{n-1} j \varepsilon_{j+1}^{(x)}\right), \quad n\geq 2,
$$
where $(\varepsilon_{j+1}^{(x)})_{j\geq 1}$ is defined in (\ref{defvarexy}).
\end{lemma}

\begin{proof}[Proof of Theorem \ref{ThmtwoERW}]
(i). Assume that $\lambda_{\alpha} \in (-1/2,0)\cup (0,1/2)$. Using (\ref{betanasy}), we see that there exists a positive constant $C(\alpha)$ such that for any $j\geq 2$,
  $$
\left|\frac{\gamma_{j-1}(\lambda_{\alpha})}{\beta_{j}(\lambda_{\alpha})} \varepsilon_{j}^{(x)}\right|  \leq C(\alpha) j^{\lambda_{\alpha}}, \quad \left|\frac{\gamma_{j-1}(-\lambda_{\alpha})}{\beta_{j}(-\lambda_{\alpha})} \varepsilon_{j}^{(y)}\right| \leq C(\alpha) j^{-\lambda_{\alpha}}.
  $$
By Lemma \ref{asymvarxy}, both $\langle M^{(x)} \rangle_n n^{-1-2\lambda_{\alpha}}$ and $\langle M^{(y)} \rangle_n n^{-1+2\lambda_{\alpha}}$ converge to positive constants almost surely as $n\to \infty$. For $j\geq 2$, we let
  $$
K_{j}:= C(\alpha) \max\left\{\frac{ j^{\lambda_{\alpha}} \sqrt{2 \log\log \sqrt{\langle M^{(x)} \rangle_j}}}{\sqrt{\langle M^{(x)} \rangle_j}}, \frac{ j^{-\lambda_{\alpha}} \sqrt{2 \log\log \sqrt{\langle M^{(y)} \rangle_j}}}{\sqrt{\langle M^{(y)} \rangle_j}}\right\},
  $$
  which, is $\FF_{j-1}$-measurable, and converges to $0$ almost surely as $j\to \infty$. By a result of Stout \cite[Theorem 1.3]{MR353428} on the law of the iterated logarithm for martingales, we have, almost surely,
  $$
\limsup_{n\to \infty} \frac{|M_n^{(x)}|}{\sqrt{2 \langle M^{(x)} \rangle_n \log \log \langle M^{(x)} \rangle_n}} \leq 1, \quad \limsup_{n\to \infty} \frac{|M_n^{(y)}|}{\sqrt{2\langle M^{(y)} \rangle_n \log \log \langle M^{(y)} \rangle_n}}\leq 1.
  $$
  Together with (\ref{Sn1Sn2xny}) and Lemma \ref{lemmaxyn}, this implies that, almost surely,
  \begin{equation}
    \label{Sn1LILproofequ}
 \limsup_{n\to \infty} \frac{|S_n^{(1)}|}{\sqrt{n \log \log n}} \leq \limsup_{n\to \infty} \frac{n \beta_n(\lambda_{\alpha})|M_n^{(x)}|+ n \beta_n(-\lambda_{\alpha})|M_n^{(y)}|}{2\sqrt{n \log \log n}} \leq C(\alpha),
  \end{equation}
  where $C(\alpha)$ is a positive constant depending on $\alpha_1$ and $\alpha_2$. This proves (\ref{Sn1LIL}).
  
  For $k=1,2,\dots,n$, we let $\FF_{n,k}:=\FF_{k+1}$ and
$$
X_{n,k}:=\frac{\sqrt{n+1}}{2}\left(\beta_{n+1}(\alpha) \frac{\gamma_k(\alpha)}{\beta_{k+1}(\alpha)} \varepsilon_{k+1}^{(x)}+\beta_{n+1}(-\alpha)\frac{\gamma_k(-\alpha)}{\beta_{k+1}(-\alpha)} \varepsilon_{k+1}^{(y)}\right).
$$
Using Lemma \ref{asymvarxy}, one has, almost surely,
\begin{equation}
  \label{less12asymvarXnk}
\begin{aligned}
  \lim_{n\to \infty}\sum_{k=1}^n \EE\left(X_{n, k}^2 \mid \mathcal{F}_{n, k-1}\right)&= \frac{1}{2}\left(1-\frac{\alpha_1}{\alpha_2}\right) + \frac{1}{4} \left(1+\frac{\alpha_1}{\alpha_2}\right)\left(\frac{1}{1-2\lambda_{\alpha}}+\frac{1}{1+2\lambda_{\alpha}}\right)\\
  &= \frac{(1+2\alpha_1^2-2\alpha_1\alpha_2)}{1-4  \alpha_1\alpha_2}.
\end{aligned}
\end{equation}
Moreover, there exists a positive constant $C(\alpha)$ such that for any $n\geq 1$ and $1\leq k \leq n$, one has $|X_{n,k}| \leq C(\alpha) n^{-\frac{1}{2}+|\lambda_{\alpha}|}$,
and thus, for any $\varepsilon>0$, almost surely, as $n\to \infty$,
$$\sum_{k=1}^n \EE\left(X_{n,k}^2 \mathds{1}_{\{|X_{n,k}|>\varepsilon\}} \mid \mathcal{F}_{n,k-1}\right) \leq \frac{C(\alpha) n^{-\frac{1}{2}+|\lambda_{\alpha}|}}{\varepsilon} \sum_{k=1}^n \EE\left(X_{n,k}^2 \mid \mathcal{F}_{n,k-1}\right) \to 0,$$
where we used (\ref{less12asymvarXnk}) in the last step. The asymptotic normality in (\ref{2ERWCLTaless12}) then follows from Proposition \ref{cltmartarray}. 

(ii). Assume that $\lambda_{\alpha} =1/2$ (the case $\lambda_{\alpha}=-1/2$ can be proved similarly). The proof for the LIL-type estimate is similar to that in (i) and we omit it here (notice that in this case, one has $\langle M^{(x)} \rangle_n \sim C_1(\alpha) n^{2}$ and $\langle M^{(y)} \rangle_n \sim C_2(\alpha) \log n$ as $n\to \infty$ where $C_1(\alpha)$ and $C_2(\alpha)$ are positive constants). 

For $k=1,2,\dots,n$, we let $\FF_{n,k}:=\FF_{k+1}$ and, by a slightly abuse of notation, let 
$$
X_{n,k}:=\frac{\sqrt{n+1}}{2\sqrt{\log (n+1)}}\left(\beta_{n+1}(\alpha) \frac{\gamma_k(\alpha)}{\beta_{k+1}(\alpha)} \varepsilon_{k+1}^{(x)}+\beta_{n+1}(-\alpha)\frac{\gamma_k(-\alpha)}{\beta_{k+1}(-\alpha)} \varepsilon_{k+1}^{(y)}\right).
$$
 Using Lemma \ref{asymvarxy}, one has, almost surely,
$$
\lim_{n\to \infty}\sum_{k=1}^n \EE\left(X_{n, k}^2 \mid \mathcal{F}_{n, k-1}\right)=\frac{\alpha_1+\alpha_2}{4\alpha_2}.
$$
Moreover, there exists a positive constant $C(\alpha)$ such that for any $n\geq 1$ and $1\leq k \leq n$, one has $|X_{n,k}| \leq \frac{C(\alpha)}{\sqrt{\log (n+1)}}$,
and thus, for any $\varepsilon>0$, almost surely, as $n\to \infty$,
$$\sum_{k=1}^n \EE\left(X_{n,k}^2 \mathds{1}_{\{|X_{n,k}|>\varepsilon\}} \mid \mathcal{F}_{n,k-1}\right) \leq \frac{C(\alpha)}{\varepsilon \sqrt{\log (n+1)}} \sum_{k=1}^n \EE\left(X_{n,k}^2 \mid \mathcal{F}_{n,k-1}\right) \to 0.$$
Again, we obtain the the asymptotic normality in (\ref{2ERWCLT12}) from Proposition \ref{cltmartarray}.

(iii). We first assume that $|\lambda_{\alpha}|\in (1/2,1)$. We only prove the case $\lambda_{\alpha}\in (1/2,1)$. The case $\lambda_{\alpha}\in (-1,-1/2)$ can be proved similarly. In view of (\ref{quadvarMxyasym}), we see that $(M^{(y)}_n)_{n\geq 1}$ is a $L^2$-bounded martingale and thus, converges almost surely to a finite limit $M_{\infty}^{(y)}$ as $n\to \infty$. Moreover, there exists a positive constant $C(\alpha)$ such that $\langle M^{(x)} \rangle_n \leq C(\alpha)n^{1+2\lambda_{\alpha}}$ for all $n\geq 1$. Choose a positive constant $\varepsilon<(\lambda_{\alpha}-1/2)/3$, by applying \cite[Theorem 1.3.15]{MR1485774} again, we have, almost surely, 
\begin{equation}
  \label{Mnxasymlln}
\lim_{n\to \infty} \frac{|M_n^{(x)}|}{n^{2\lambda_{\alpha}}}\leq \lim_{n\to \infty} \frac{\langle M^{(x)} \rangle_n^{\frac{1}{2}+\varepsilon}}{n^{2\lambda_{\alpha}}} \leq \lim_{n\to \infty} C(\alpha) n^{\frac{1}{2}-\lambda_{\alpha}+3\varepsilon }=0,
\end{equation}
which, together with (\ref{betanasy}) and (\ref{Sn1Sn2xny}), implies that, almost surely,
$$
\lim_{n\to \infty} \frac{S_n^{(1)}}{n^{\lambda_{\alpha}}}=\lim_{n\to \infty}\frac{n^{1-\lambda_{\alpha}}}{2} \left(\beta_n(\lambda_{\alpha}) (x_1+ M_n^{(x)})+\beta_n(-\lambda_{\alpha}) (y_1+M_n^{(y)})\right) = \frac{y_1+M_{\infty}^{(y)}}{2\Gamma(1+\lambda_{\alpha})} , 
$$
and similarly,
$$\lim_{n\to \infty} \frac{S_n^{(2)}}{n^{\lambda_{\alpha}}}=\lim_{n\to \infty}\frac{n^{1-\lambda_{\alpha}}}{2r_{\alpha}} \left(\beta_n(-\lambda_{\alpha}) (y_1+ M_n^{(y)})-\beta_n(\lambda_{\alpha}) (x_1+M_n^{(x)})\right) = \frac{y_1+M_{\infty}^{(y)}}{2r_{\alpha}\Gamma(1+\lambda_{\alpha})}.$$
We set $W:=(y_1+M_{\infty}^{(y)})/(2\Gamma(1+\lambda_{\alpha}))$. By definition, $\EE (\varepsilon_{j+1}^{(y)})^2>0$ for all $j\geq 1$. Thus, 
$$
\operatorname{Var}(M_{\infty}^{(y)})=\sum_{j=1}^{\infty} \left(\frac{\gamma_j(-\lambda_{\alpha})}{\beta_{j+1}(-\lambda_{\alpha})}\right)^2 \EE((\varepsilon_{j+1}^{(y)})^2)>0,
$$
which shows that $W$ is not degenerate. 

 We now assume that $\lambda_{\alpha}\in \{-1, 1\}$. To be more instructive, we prove the case $\lambda_{\alpha}=-1$ (that is, $\alpha_1=\alpha_2=-1$). The case $\lambda_{\alpha}=1$ can be proved similarly. As in (\ref{Mnxasymlln}), we can apply Lemma \ref{lemlambda1} (i) and the law of large numbers for martingales to prove that $y_n\to 0$ almost surely. Thus, almost surely,
$$
\lim_{n\to \infty}\frac{S_n^{(1)}}{n}=-\lim_{n\to \infty}\frac{S_n^{(2)}}{n}=\frac{1}{2}\lim_{n\to \infty} x_n.
$$
Since $(x_n)_{n\geq 1}$ is a bounded martingale, one has
$
\operatorname{Var}(\lim_{n\to \infty} x_n)=\sum_{j=2}^{\infty} \EE(x_{j+1}-x_j)^2>0$ (note that $x_2-x_1=0$), and in particular, $\lim_{n\to \infty} x_n$ is not degenerate. This completes the proof.
\end{proof}

To prove Theorem \ref{Thmal2difsin2ndm}, we first prove an auxiliary lemma.

\begin{lemma}
  \label{lemsumGammanjsqr}
Assume that $\alpha_1\alpha_2<0$. Recall that $\lambda_{\alpha} =\mathrm{i}\sqrt{-\alpha_1\alpha_2}$. Then, as $n\to \infty$,
$$
\sum_{j=2}^n\frac{\Gamma^2(j)}{\Gamma^2(j\pm\lambda_{\alpha})} = \frac{\Gamma^2(n+1)}{(1\mp 2\lambda_{\alpha})n\Gamma^2(n\pm \lambda_{\alpha})} + O(\log n).
$$
\end{lemma}
\begin{proof}
  Notice that for $j\geq 2$,
$$
\begin{aligned}
\frac{\Gamma^2(j+1)}{j\Gamma^2(j\pm\lambda_{\alpha})}-\frac{\Gamma^2(j)}{(j-1)\Gamma^2(j-1\pm\lambda_{\alpha})}
  &=\frac{\Gamma^2(j)}{\Gamma^2(j\pm\lambda_{\alpha})}\left(j-\frac{(j-1\pm \lambda_{\alpha})^2}{j-1}\right) \\
  &=\frac{\Gamma^2(j)}{\Gamma^2(j\pm\lambda_{\alpha})}\left(1\mp 2\lambda_{\alpha} -\frac{\alpha_1\alpha_2}{j-1}\right),
\end{aligned}
$$
Since $\Gamma(j)/\Gamma(j\pm\lambda_{\alpha})$ are uniformly bounded for all $j\geq 2$, this implies that as $j\to \infty$,
$$\frac{\Gamma^2(j)}{\Gamma^2(j\pm\lambda_{\alpha})} = \frac{1}{1\mp 2\lambda_{\alpha}}\left(\frac{\Gamma^2(j+1)}{j\Gamma^2(j\pm\lambda_{\alpha})}-\frac{\Gamma^2(j)}{(j-1)\Gamma^2(j-1\pm\lambda_{\alpha})}\right)+O\left(\frac{1}{j}\right).$$
The desired result then follows from telescoping the above equality over $j$ from $2$ to $n$.
\end{proof}

The proof of Theorem \ref{Thmal2difsin2ndm} is similar to that of Theorem \ref{ThmtwoERW}, so we only highlight the main differences.

\begin{proof}[Proof of Theorem \ref{Thmal2difsin2ndm}]
  Recall $\varepsilon_{j+1}^{(x)}$ and $\varepsilon_{j+1}^{(y)}$ defined in (\ref{defvarexy}) and $\gamma_j(\cdot)$ and $\beta_j(\cdot)$ defined in (\ref{defgammabeta}). Observe that $M^{(x)}$ and $M^{(y)}$ defined in (\ref{defMxy}) are (complex-valued) martingales with bounded increments. Thus, by applying the law of the iterated logarithm for martingales \cite{stout1970martingale} to their real and imaginary parts separately, one has, almost surely,
$$
\limsup_{n\to \infty} \frac{|M_n^{(x)}|}{\sqrt{n \log \log  n }} \leq C(\alpha), \quad \limsup_{n\to \infty} \frac{|M_n^{(y)}|}{\sqrt{n \log \log  n}} \leq C(\alpha),
$$
where $C(\alpha)$ is a positive constant depending on $\alpha_1$ and $\alpha_2$. Then using the same argument as in (\ref{Sn1LILproofequ}), we obtain that, almost surely,
$$
\limsup_{n\to \infty} \frac{|S_n^{(1)}|}{\sqrt{n \log \log n}} \leq C(\alpha),
$$
which proves (\ref{Sn1LILnega}). 

For $k=1,2,\dots,n$, we let $\FF_{n,k}:=\FF_{k+1}$ and, by a slightly abuse of notation, let
$$
X_{n,k}=\frac{\sqrt{n+1}}{2} \left(\frac{\beta_{n+1}(\lambda_{\alpha})\gamma_k(\lambda_{\alpha})}{\beta_{k+1}(\lambda_{\alpha})} \varepsilon_{k+1}^{(x)}+\frac{\beta_{n+1}(-\lambda_{\alpha}) \gamma_k(-\lambda_{\alpha})}{\beta_{k+1}(-\lambda_{\alpha})} \varepsilon_{k+1}^{(y)}\right), \quad 1\leq k \leq n.
$$
For a complex number $z$, we denote its complex conjugate by $\bar{z}$. It is clear from the definitions (\ref{defgammabeta}), (\ref{defxyn}) and (\ref{defvarexy}) that 
$$
\bar{x}_k=y_k,\quad \overline{\varepsilon_{k+1}^{(x)}}=\varepsilon_{k+1}^{(y)}, \quad \overline{\gamma_k(\lambda_{\alpha})}=\gamma_k(-\lambda_{\alpha}), \quad \overline{\beta_k(\lambda_{\alpha})}=\beta_k(-\lambda_{\alpha}).
$$
In particular, $(X_{n,k})_{1\leq k \leq n}$ defined above is a real-valued martingale difference array. By (\ref{betanasy}) and Lemma \ref{asymvarxy}, almost surely, 
  \begin{equation}
    \label{limjvarexy}
\lim_{k\to \infty} \frac{ (n+1)^2\beta_{n+1}(\lambda_{\alpha})\gamma_k(\lambda_{\alpha})}{\beta_{k+1}(\lambda_{\alpha})} \frac{\beta_{n+1}(-\lambda_{\alpha}) \gamma_k(-\lambda_{\alpha})}{\beta_{k+1}(-\lambda_{\alpha})} \EE ( \varepsilon_{k+1}^{(x)} \varepsilon_{k+1}^{(y)} | \FF_k) = 1-\frac{\alpha_1}{\alpha_2}.
  \end{equation}
Using (\ref{defgammabeta}), Lemma \ref{asymvarxy} and Lemma \ref{lemsumGammanjsqr}, one has, almost surely, as $n\to \infty$,
\begin{equation}
  \label{condvarxdiff}
\begin{aligned}
  &\quad\ (n+1)^2\beta_{n+1}^2(\lambda_{\alpha})\sum_{k=1}^n\left(\frac{ \gamma_k(\lambda_{\alpha})}{\beta_{k+1}(\lambda_{\alpha})}\right)^2 \EE((\varepsilon_{k+1}^{(x)})^2 \mid \FF_k) \\
  &=\frac{\Gamma^2(n+1- \lambda_{\alpha})}{ \Gamma^2(n+1)}\sum_{k=1}^n(1+\lambda_{\alpha})^2\EE((\varepsilon_{k+1}^{(x)})^2 \mid \FF_k) \frac{\Gamma^2(k+1)}{\Gamma^2(k+1-\lambda_{\alpha})} \\
  &=\frac{(\alpha_1+\alpha_2)n}{\alpha_2(1-2\lambda_{\alpha})} + o(n),
\end{aligned}
\end{equation}
and similarly,
\begin{equation}
  \label{condvarydiff}
 ((n+1)\beta_{n+1}(\lambda_{\alpha}))^2\sum_{k=1}^n\left(\frac{ \gamma_k(\lambda_{\alpha})}{\beta_{k+1}(\lambda_{\alpha})}\right)^2 \EE((\varepsilon_{k+1}^{(x)})^2 \mid \FF_k)=\frac{(\alpha_1+\alpha_2)n}{\alpha_2(1+2\lambda_{\alpha})}+ o(n).
\end{equation}
Similarly as in (\ref{less12asymvarXnk}), by using (\ref{limjvarexy}), (\ref{condvarxdiff}) and (\ref{condvarydiff}), we obtain that, almost surely, 
\begin{equation}
  \lim_{n\to \infty}\sum_{k=1}^n \EE\left(X_{n, k}^2 \mid \mathcal{F}_{n, k-1}\right)= \frac{1+2\alpha_1^2-2\alpha_1\alpha_2}{1-4  \alpha_1\alpha_2}.
\end{equation}
The rest of the proof for first convergence in (\ref{2ERWCLTa12diff}) then follows the same lines as that of (\ref{2ERWCLTaless12}) (we note that there exists a positive constant $C(\alpha)$ such that $|\sqrt{n} X_{n, k} |\leq C(\alpha)$ for all $n\geq k\geq 1$). The second convergence in (\ref{2ERWCLTa12diff}) can be proved similarly. .
\end{proof}

\section{Proof of Theorem \ref{flucthm}}
\label{secthmfluc}
Recall $\hat{\beta}_n$, $\beta_n(\cdot)$ and $\gamma_n(\cdot)$ defined in Section \ref{secthm2dERW}. We first prove an auxiliary lemma.

\begin{lemma}
  \label{lemmaXnksquarecon}
(i). Assume that $\lambda_{\alpha}\in (1/2,1)$. For any $n\geq 1$ and $k\geq 1$, let 
   \begin{equation}
  \label{defXnkfluc}
 X_{n,k}:=\begin{cases}
   \frac{\sqrt{n}\beta_n(\lambda_{\alpha}) }{2} \frac{\gamma_{k}(\lambda_{\alpha})}{\beta_{k+1}(\lambda_{\alpha})} \varepsilon_{k+1}^{(x)}, & 1\leq k < n, \\
   -\frac{n^{\lambda_{\alpha}-\frac{1}{2}}}{2\Gamma(1+\lambda_{\alpha})} \frac{\gamma_{k}(-\lambda_{\alpha})}{\beta_{k+1}(-\lambda_{\alpha})} \varepsilon_{k+1}^{(y)}, & k\geq n,
 \end{cases}
 \end{equation}
 where $\varepsilon_{k+1}^{(x)}$ and $\varepsilon_{k+1}^{(y)}$ are defined be (\ref{defvarexy}). Then,  $\sup_{n\geq 1} \sum_{k=1}^{\infty}\EE X^2_{n,k}<\infty$, and
 $$
\lim_{n\to \infty}\sum_{k=1}^{\infty} X_{n,k}^2  =\left(1+\frac{\alpha_1}{\alpha_2}\right)\frac{\sqrt{\alpha_1\alpha_2}}{4\alpha_1\alpha_2-1} \quad \text{ in probability}.
  $$
  (ii). Assume that $\lambda_{\alpha}=1$ $($i.e. $\alpha_1=\alpha_2=1)$. For any $n\geq 2$ and $k\geq 2$, let 
    \begin{equation}
  \label{defXnkfluclambda1} 
 \widetilde{X}_{n,k}:=\begin{cases}
   \frac{\sqrt{n}\hat{\beta}_n}{2} k \varepsilon_{k+1}^{(x)}, & 2\leq k < n, \\
   -\frac{\sqrt{n}}{2} (y_{k+1}-y_{k}), & k\geq n.
 \end{cases}
 \end{equation}
 Then, $\sup_{n\geq 2}  \sum_{k=2}^{\infty}\EE \widetilde{X}^2_{n,k}<\infty$, and
 $$\lim_{n\to \infty}\sum_{k=2}^{\infty} \widetilde{X}_{n,k}^2  =\frac{2-2W^2}{3} \quad \text{ in probability},$$  
 where $W$ is as in Theorem \ref{ThmtwoERW} (iii).
\end{lemma}
\begin{proof}
(i). Recall that we use $C(\alpha)$ to denote a positive constant depending on $\alpha_1$ and $\alpha_2$ whose value may change from line to line. By definition (\ref{defvarexy}), we have $|\varepsilon^{(x)}_{k+1}| \leq C(\alpha)$ and $|\varepsilon^{(y)}_{k+1}| \leq C(\alpha)$ for all $k\geq 1$. Using (\ref{betanasy}), one can easily check that 
$$ \sup_{n\geq 1}\sum_{k=1}^{\infty}\EE X^2_{n,k} \leq C(\alpha) \sup_{n\geq 1} \left(n\beta^2_n(\lambda_{\alpha}) \sum_{k=1}^{n-1}\frac{\gamma^2_{k}(\lambda_{\alpha})}{\beta^2_{k+1}(\lambda_{\alpha})}+ n^{2\lambda_{\alpha}-1} \sum_{k=n}^{\infty} \frac{\gamma^2_{k}(-\lambda_{\alpha})}{\beta^2_{k+1}(-\lambda_{\alpha})}\right)<\infty.$$
By Lemma \ref{asymvarxy}, we have, almost surely,
  \begin{equation}
    \label{asymcondvarXnkfluc}
\begin{aligned}
    \lim_{n\to \infty}\sum_{k=1}^{\infty}\EE  (X_{n,k}^2 \mid \FF_{k})&=  \lim_{n\to \infty}\frac{n\beta^2_n(\lambda_{\alpha}) }{4}\sum_{k=1}^{n-1}\frac{\gamma^2_{k}(\lambda_{\alpha})}{\beta^2_{k+1}(\lambda_{\alpha})} \EE ((\varepsilon_{k+1}^{(x)})^2 \mid \FF_k )\\
  &\quad +  \lim_{n\to \infty}\frac{n^{2\lambda_{\alpha}-1}}{4\Gamma^2(1+\lambda_{\alpha})} \sum_{k=n}^{\infty} \frac{\gamma^2_{k}(-\lambda_{\alpha})}{\beta^2_{k+1}(-\lambda_{\alpha})} \EE ((\varepsilon_{k+1}^{(y)})^2 \mid \FF_k ) \\ 
  &=\frac{1}{4}\left(1+\frac{\alpha_1}{\alpha_2}\right)\left(\frac{1}{2\lambda_{\alpha}-1}+\frac{1}{2\lambda_{\alpha}-1}\right) = \left(1+\frac{\alpha_1}{\alpha_2}\right)\frac{\sqrt{\alpha_1\alpha_2}}{4\alpha_1\alpha_2-1}.
  \end{aligned} 
  \end{equation}
 On the other hand, for any $\varepsilon>0$, by Markov's inequality, 
  \begin{equation}
    \label{Xnksquareconv1}
\PP(\left|\sum_{k=1}^{n-1} (X_{n,k}^2- \EE(X_{n,k}^2 \mid \FF_k))\right|\geq \varepsilon) \leq \frac{\sum_{k=1}^{n-1} \EE X_{n,k}^4}{\varepsilon^2} \leq \frac{C(\alpha)}{\varepsilon^2 n} \to 0, \quad \text{ as } n\to \infty,
  \end{equation}
  where we used (\ref{betanasy}) and that $\varepsilon_{k+1}^{(x)}$ is uniformly bounded in $k$ in the second inequality. Similarly, one can show that, for any $\varepsilon>0$, 
  $$\PP(\left|\sum_{k=n}^{\infty} (X_{n,k}^2- \EE(X_{n,k}^2 \mid \FF_k))\right|\geq \varepsilon) \leq \frac{C(\alpha)}{\varepsilon^2 n} \to 0, \quad \text{ as } n\to \infty,$$
  which, together with (\ref{asymcondvarXnkfluc}) and (\ref{Xnksquareconv1}), implies the desired result. 

  (ii). From the proof of Theorem \ref{ThmtwoERW} (iii), we see that $x_n\to 0$ almost surely and that both $S^{(1)}_n/n$ and $S^{(2)}_n/n$ converge almost surely to $W=\lim_{n\to \infty}y_n/2$. As shown in (\ref{recudeltxn}), for any $k\geq 2$, one has 
  $$y_{k+1}-y_k=\frac{1}{k+1}(X_{k+1}^{(1)}+X_{k+1}^{(2)}-y_k).$$ 
One can then deduce from (\ref{defvarexy}) and (\ref{condvarxnyn}) that almost surely,
  $$\lim_{k\to \infty} k^2\EE((y_{k+1}-y_k)^2 \mid \FF_k)=2-2W^2, \quad \lim_{k\to \infty} \EE((\varepsilon^{(x)}_{k+1} )^2 \mid \FF_k)=\frac{1-W^2}{2}.$$
 Then, by using the same argument as in (i), we obtain the desired result.
\end{proof}

We are now in a position to prove Theorem \ref{flucthm}.

\begin{proof}[Proof of Theorem \ref{flucthm}]
  Without loss of generality, we assume that $\lambda_{\alpha}>1/2$. The case $\lambda_{\alpha}<-1/2$ can be proved similarly. \\
  (i). We first assume that $\lambda_{\alpha}\in (1/2,1)$. For $n\geq 1$ and $k\geq 1$, we let $S_{n,k}:=\sum_{j=1}^{k} X_{n,j}$ where $(X_{n,j})_{n\geq 1, j\geq 1}$ are as in (\ref{defXnkfluc}), and let $\FF_{n,k}:=\FF_{k+1}$. Then for each $n\geq 1$, the sequence $(S_{n,k}, \FF_{n,k})_{k\geq 1}$ is a square-integrable martingale. In particular, for any $n\geq 1$,
  $$
S_{n,\infty}:=\lim_{k\to \infty} S_{n,k}=\frac{\sqrt{n}\beta_n(\lambda_{\alpha}) }{2}\sum_{j=1}^{n-1} \frac{\gamma_j(\lambda_{\alpha})}{\beta_{j+1}(\lambda_{\alpha})} \varepsilon_{j+1}^{(x)}-\frac{n^{\lambda_{\alpha}-\frac{1}{2}}}{2\Gamma(1+\lambda_{\alpha})} \sum_{j=n}^{\infty} \frac{\gamma_j(-\lambda_{\alpha})}{\beta_{j+1}(-\lambda_{\alpha})} \varepsilon_{j+1}^{(y)}
  $$
  exists. From (\ref{Sn1Sn2xny}), Lemma \ref{lemmaxyn} and the proof of Theorem \ref{ThmtwoERW} (iii), we see that 
  $$
  \begin{aligned}
    \frac{S_n^{(1)}}{n^{\lambda_{\alpha}}}-W&= \frac{n^{1-\lambda_{\alpha}}\beta_n(\lambda_{\alpha})}{2}(x_1 +M_n^{(x)}) + \frac{n^{1-\lambda_{\alpha}}\beta_n(-\lambda_{\alpha})}{2}(y_1+M_n^{(y)})-\frac{y_1+M_{\infty}^{(y)}}{2\Gamma(1+\lambda_{\alpha})} \\
    &=\frac{n^{1-\lambda_{\alpha}}\beta_n(\lambda_{\alpha})}{2}M_n^{(x)} - \frac{M_{\infty}^{(y)}-M_n^{(y)}}{2\Gamma(1+\lambda_{\alpha})}  +R_n,
  \end{aligned}
  $$
  where 
  $$
 R_n:=\frac{n^{1-\lambda_{\alpha}}\beta_n(\lambda_{\alpha})}{2}x_1+ \left(\frac{n^{1-\lambda_{\alpha}}\beta_n(-\lambda_{\alpha})}{2}-\frac{1}{2\Gamma(1+\lambda_{\alpha})}\right) (y_1+M_{\infty}^{(y)}).
$$
 By the properties of the Gamma function, we have that, as $n\to \infty$,
 $$n^{1-\lambda_{\alpha}}\beta_n(-\alpha)-\frac{1}{\Gamma(1+\lambda_{\alpha})}=O\left(\frac{1}{n}\right).$$  
 And in particular, for $n\geq 1$, we can write 
 $$
  n^{\lambda_{\alpha}-\frac{1}{2}}\left(\frac{S_n^{(1)}}{n^{\lambda_{\alpha}}}-W\right) =S_{n,\infty} + n^{\lambda_{\alpha}-\frac{1}{2}} R_n,
$$
where $n^{\lambda_{\alpha}-\frac{1}{2}} R_n \to 0$ almost surely as $n\to \infty$. We note that, by (\ref{betanasy}), there exists a positive constant $C(\alpha)$ such that for any $n,k\geq 1$, one has $|X_{n,k}| \leq C(\alpha) n^{-\frac{1}{2}}$. In particular, $\sup_k|X_{n,k}|\to 0$ in probability as $n\to \infty$ and $\EE\left( \sup_k X^2_{n,k}\right)$ is bounded in $n$. 
Then the desired result follows from Lemma \ref{lemmaXnksquarecon} (i) and Slutsky's theorem. 

(ii). Now we assume that $\lambda_{\alpha}=1$. Recall from the proof of Lemma \ref{lemmaXnksquarecon} (ii) that $W=\lim_{n\to \infty}y_n/2$. Thus, by (\ref{Sn1Sn2xny}) and Lemma \ref{lemlambda1} (ii), for $n>2$, one has
$$
\sqrt{n}\left(\frac{S_n^{(1)}}{n}-W\right)=\frac{\sqrt{n}\hat{\beta}_n}{2}\sum_{k=2}^{n-1} j\varepsilon_{k+1}^{(x)}-\frac{\sqrt{n}}{2}\sum_{k=n}^{\infty}(y_{k+1}-y_k)+\frac{\sqrt{n}\hat{\beta}_n}{2} x_2.
$$
The rest of the proof is similar to that of (i). The only difference is that we use Lemma \ref{lemlambda1} (ii) and Lemma \ref{lemmaXnksquarecon} (ii) instead of Lemma \ref{lemmaxyn} and Lemma \ref{lemmaXnksquarecon} (i). We omit the details here.  
\end{proof}

\section{Acknowledgments}

Rafik Aguech is supported by Ongoing Research Funding program, (ORF-2025-987), King Saud University, Riyadh, Saudi Arabia. Shuo Qin is supported by the China Postdoctoral Science Foundation under Grant Number 2025M773086.

\bibliographystyle{plain}
\bibliography{math_ref}

\end{document}